\documentclass[a4paper,11pt]{amsart}

\usepackage[english]{babel} 

\usepackage[
hmarginratio={1:1},     
vmarginratio={1:1},     
textwidth=16cm,        
textheight=21cm,
heightrounded,          
]{geometry}

\usepackage[T1]{fontenc} 
\usepackage[utf8]{inputenc} 
\usepackage{lmodern}
\usepackage{microtype}

 \usepackage[colorlinks=true]{hyperref}
\hypersetup{urlcolor=blue, citecolor=red}

\numberwithin{equation}{section}

\newtheorem{theorem}{Theorem}[section]
\newtheorem{lemma}[theorem]{Lemma}
\newtheorem{proposition}[theorem]{Proposition}

\theoremstyle{remark}
\newtheorem{remark}[theorem]{Remark}
\theoremstyle{definition}
\newtheorem{definition}[theorem]{Definition}

\newcommand{\abs}[1]{\lvert#1\rvert}

\newcommand{\R}{\mathbb{R}}
\newcommand{\N}{\mathbb{N}}

\newcommand{\cN}{{\mathcal N}}

\let\ge=\varepsilon

\begin{document}

\title[Concentration phenomena]{Concentration phenomena for the
  Schr\"odinger-Poisson system in \protect$\R^2$}

\author{Denis Bonheure \and Silvia Cingolani \and Simone Secchi}

\address[D. Bonheure]{D\'epartement de Math\'ematiques
	\newline\indent
	Universit\'e Libre de Bruxelles, \newline\indent CP 214, Boulevard du Triomphe, B-1050 Bruxelles, Belgium}
\email{denis.bonheure@ulb.ac.be}

\address[S. Cingolani]{Dipartimento di Matematica, \newline\indent
  Universit\`a degli Studi di Bari Aldo Moro, \newline \indent Via
  Orabona 4, 70125 Bari, Italy} \email{silvia.cingolani@uniba.it}

\address[S. Secchi]{Dipartimento di Matematica e Applicazioni
	\newline\indent
        Universit\`a degli Studi di Milano Bicocca, \newline \indent Via Roberto Cozzi 55, 20125 Milano, Italy}
\email{simone.secchi@unimib.it}


\begin{abstract}
  We perform a semiclassical analysis for the planar
  Schr\"odin\-ger-Poisson system
  \begin{gather}%
    \begin{cases}
      -\varepsilon^{2} \Delta\psi+V(x)\psi= E(x) \psi \quad \text{in $\R^2$},\\
      -\Delta E= |\psi|^{2} \quad \text{in $\R^2$},
    \end{cases} \tag{$SP_\varepsilon$}
  \end{gather}
  where $\varepsilon$ is a positive parameter corresponding to the
  Planck constant and $V$ is a bounded external potential.  We
  detect solution pairs~$(u_\varepsilon, E_\varepsilon)$ of the system
  $(SP_\varepsilon)$ as~$\ge \rightarrow 0$, leaning on a nongeneracy
  result in \cite{BCVS}.
\end{abstract}


\maketitle

\section{Introduction}

We are concerned with the planar Schr\"{o}dinger-Poisson system
\begin{align}%
  \begin{cases}
    -\varepsilon^{2} \Delta\psi+V(x)\psi= E(x) \psi \quad \text{in $\R^2$},\\
    -\Delta E= |\psi|^{2} \quad \text{in $\R^2$},
  \end{cases}
  \label{sys:s-n}%
\end{align}
which presents some special features, because of the different nature
of the Newtonian potential in two-dimensional space.  This system has
been derived in \(\mathbb{R}^3\) by R. Penrose in \cite{pe2} in his
description of the self-gravitational collapse of a quantum mechanical
system (see also \cite{pe1,pe3,mt,mt0}).  The rigorous mathematical
study of the nonlinear Schr\"odinger equation with nonlocal
nonlinearity, involving a Coulomb type convolution potential, dates
back to the seminal papers by Lieb \cite{lieb} and Lions
\cite{l2}. Successively in \cite{WW} Wei and Winter studied the
semiclassical limit for the Schr\"odinger-Poisson system, after
showing the nondegeneracy of the least energy solutions of a related
limiting system (see also \cite{L}). We also mention the papers
\cite{CCS,CSS,CT,MVS} where variational and topological methods have
been employed to derive concentration phenomena for generalized NLS
equations with more general nonlocal nonlinearity in dimensional
$d \geq 3$, where the nondegeneracy properties of the linearized
operators do not hold.

The rigorous study of the Schr\"{o}dinger-Poisson system in $\R^2$
remained open for long time, since it appears more delicate.
Differently from the Coulomb potential, the Newton potential in $\R^2$
is sign-changing and it presents singularities at zero and
infinity. Moreover we recall that the Poisson equation
$-\Delta E= |\psi|^{2}$ determines the solution~$E\colon \R^2 \to \R$
only up to harmonic functions, and every semibounded harmonic function
is costant in $\R^2$.  Therefore if $\psi \in L^\infty(\R^2)$ and $E$
solves the Poisson equation under suitable additional assumption at
infinity, such as $E(x) \to - \infty$ as $|x| \to + \infty$, then we
have $E(x)= \Phi_\psi(x) +c$, where $c$ is a constant and $\Phi_\psi$
is the convolution of fundamental solution of $-\Delta$ in $\R^2$ with
$|\psi|^2$, namely
$\Phi_\psi(x)= \frac{1}{2 \pi} \int_{\R^2} \log \frac{1}{|x-y|}
|\psi(y)|^2 dy$.

In literature, apart from some numerical results in \cite{HMT},
existence and uniqueness results of spherically symmetric solutions of
\eqref{sys:s-n} were proved by Stubbe and Vuffray \cite{CSV}, for $V \equiv 1$, using
shooting methods for the associated ODE system (see also \cite{CS} for
the one-dimensional case).

In \cite{M} Masaki proved a global well-posedness of the Cauchy
problem for \eqref{sys:s-n} in a subspace of $H^1(\R^2)$,
where~$E(x)=\frac{1}{2 \pi} \int_{\R^2} \log \frac{|y|}{|x-y|}
|\psi(y)|^2 dy$, which means \(E(0)=0\).

\medskip

In the more natural case, \(E\) coincides with the Newtonian potential
\(\Phi\) of \(|\psi|^2\), the Schr\"odinger-Poisson system with a
constant potential can be written as the following Schr\"odinger
equation with a nonlocal nonlinearity:
\begin{equation} \label{key56log2bis3}
  - \Delta u + u = \frac{1}{2
    \pi} \left[\log \frac{1}{|\cdot|} \star \abs{u}^2 \right] u ,
  \quad x \in \R^2.
\end{equation} 
For such an integro-differential equation, unlike the 3D case, the
applicability of variational tools is not straightforward, because the
usual Sobolev spaces do not provide a good environment to work in.  In
\cite{Stubbe} Stubbe tackled this problem by setting a suitable
variational framework for \eqref{key56log2bis3} within the space
\begin{gather*}
  X= \left\{u \in H^1(\R^2)\mid \int_{\R^2}\log(1+|x|) |u(x)|^2 \,dx <
    \infty \right\},
\end{gather*}
endowed with the norm
\begin{align*}
  \|u\|^2_X= \int_{\mathbb{R}^2} \left( |\nabla u|^2 + |u|^2 \right) \, dx + \int_{\R^2}\log(1+|x|) |u(x)|^2 \,dx.
\end{align*}
The space $X$ provides a reasonable variational framework, but its
norm does not detect the invariance of the problem under translations;
furthermore the quadratic part of the energy functional associated to
\eqref{key56log2bis3} is not coercive on \(X\).  These difficulties
enforced the implementation of new variational ideas and estimates to
treat nonlinear Schr\"odinger equation with nonlocal nonlinearities
involving logarithmic type convolution potential \cite{CJ,CW,DW}. In
particular in \cite{CW}, the authors proved the existence result of an
unique positive ground state solution $U$ to
\eqref{key56log2bis3}. Sharp asymptotics and the nondegeneracy of the
ground state solution $U$ has been proved in \cite{BCVS}.

In the present paper we study the existence of solution pairs of the
Schr\"odinger-Poisson system as the parameter
$\varepsilon \rightarrow 0^+$.  This study presents some new aspects
with respect to the 3D case, since the Newtonian potential in $\R^2$
does not scale algebraically.

The semiclassical analysis remained in the background until very
recent years and, to the best of our knowledge, it has only been
treated by Masaki in~\cite{M} via WKB approximation.

Here we adapt some pertubation method developed in \cite{ABC,AM} in
the variational framework $X$ where the norm depends on
the weight~$x \mapsto \log(1+ |x|)$.  This makes it more involved to apply a finite
dimensional reduction.

\bigskip

In the rest of the paper we will consider a potential
function $V \colon \mathbb{R}^2 \to \mathbb{R}$ satisfying the
following condition:
\begin{itemize}
\item[(V)] $V \in \mathrm{C}^2(\mathbb{R}^2)$, $\inf_{x \in \mathbb{R}^2} V(x) >0$ and
  \begin{gather*}
  \sup_{x \in \mathbb{R}^2} \left[
   \left| V(x) \right| + \sum_{j=1}^2 \left| \partial_j V(x) \right| + \sum_{i,j=1}^2 \left|\partial^2_{ij} V(x) \right| \right] <+\infty.
  \end{gather*}
\end{itemize}

\bigskip

Setting $v(x) = {\varepsilon} \psi(x)$, the system~\eqref{sys:s-n} can
be written
\begin{align}%
  \begin{cases}
    -\varepsilon^{2}\Delta v +{V}(x)v = E v
    \quad \text{in $\R^2$}, \\
    -\varepsilon^{2}\Delta{E}= |v|^{2} \quad \text{in $\R^2$}.
  \end{cases}
  \label{sys:s-n-red90}%
\end{align}
Our main existence result can be summarized as follows.
\begin{theorem} \label{th:main}
  Suppose that $V$ satisfies (V) and has a non-degenerate critical
  point $x_0$, i.e. $\nabla V(x_0)=0$ and $D^2 V(x_0)$ is either
  positive- or negative-definite. Then, for every $\ge>0$ sufficiently
  small, the system \eqref{sys:s-n-red90} possesses a solution
  $(v_\varepsilon, E_\varepsilon)$ such that
	\begin{gather*}
          v_\ge(x) \simeq U \left(\frac{x-x_0}{\varepsilon}\right),
          \quad E_\varepsilon (x)= \frac{1}{\varepsilon^2}\int_{\R^2}
          \log \frac{\varepsilon}{|x-z|} |v_\varepsilon(z)|^2 \,dz
	\end{gather*}
	where $U$ is the unique (up to translations) positive ground
        state solution of the limiting equation
	\begin{gather} \label{key56log2bis44}
          - \Delta u + V(x_0) u = \frac{1}{2 \pi} \left[\log
            \frac{1}{|\cdot|} \star \abs{u}^2 \right] u , \quad x \in
          \R^2.
	\end{gather}
\end{theorem}

\bigskip

\begin{remark}
  In Theorem \eqref{th:main} we have
  $E_\varepsilon(x) = \varepsilon^{-2} \Phi_{v_\varepsilon}(x)
  + c_\varepsilon$ where
  $\Phi_{v_\varepsilon}(x) = \log \frac{1}{|\cdot |} \star
  v_\varepsilon^2$ and
  $c_\varepsilon = \varepsilon^{-2} \log \varepsilon\|
  v_\varepsilon\|^2_2$.  Coming back to the system~\eqref{sys:s-n}, we
  derive the existence of the solution pair
  $\left(\varepsilon^{-1} v_\varepsilon, E_\varepsilon \right)$ for
  $\varepsilon >0$ small.
\end{remark}

\section{Functional setting}

Without loss of generality, we will assume that $x_0=0$ and $V(0)=1$.
Setting $u(x)= v(\varepsilon x)$ and $\omega (x)= {E}(\ge x)$, the
system \eqref{sys:s-n-red90} becomes
\begin{align}%
  \begin{cases}
    -\Delta u+ V(\varepsilon x) u = \omega(x)  u  \quad \quad \text{in $\R^2$},\\
    -\Delta \omega = |u|^{2} \quad \quad \text{in $\R^2$}.
  \end{cases}
  \label{sys:s-n-red2}%
\end{align}
The second equation in~\eqref{sys:s-n-red2} can be explicitly solved
with respect to $\omega$. Choosing \(\omega\) as the convolution of
the fundamental solution of \(-\Delta\) in \(\R^2\) with \(|u|^2\),
this system can be written as the single nonlocal equation
\begin{align} \label{key56}
  - \Delta u + V(\varepsilon x) u =
  \frac{1}{2 \pi} \left[\log \frac{1}{|\cdot|} \star \abs{u}^2 \right]
  u , \quad x \in \R^2.
\end{align}
We consider the functional space
\begin{align*}
	X = \left\{ u \in H^1(\mathbb{R}^2) \mid |u|_{*}<+\infty \right\},
\end{align*}
where
\begin{align*}
	\left| u \right|_{*}^2 = \int_{\mathbb{R}^2} \log
	\left(1+|x| \right) |u(x)|^2 \, dx.
\end{align*}
We endow \(X\) with the norm
\begin{gather*}
	\|u\|_X^2 = \|u\|_{H^1}^2 + |u|_{*}^2
\end{gather*}
and the associated scalar product
\begin{gather*}
  \langle u \mid v \rangle_X = \int_{\mathbb{R}^2} \left[ \nabla u \cdot
    \nabla v + uv \right] \, dx + \int_{\mathbb{R}^2} \log (1+|x|)
  u(x) v(x) \, dx.
\end{gather*}
The norms in $H^1(\mathbb{R}^2)$ and $L^{q}(\mathbb{R}^2)$ will be denoted by $\Vert \cdot \Vert_{H^1}$ and $|\cdot |_q$, respectively.

It is known that the space $X$ is compactly embedded in $L^p(\R^2)$ for any $p \in [2, +\infty)$ (cf. Lemma 2.2 in \cite{CW}).
 
Solutions to \eqref{key56} correspond to critical points of the energy
functional \(I_\varepsilon \colon X \to \mathbb{R}\)
defined by
\begin{gather*} 
  I_\varepsilon(u) = \frac{1}{2} \int_{\mathbb{R}^2} |\nabla u|^2 +
  V_\varepsilon |u|^2 \, dx 
  - \frac{1}{8\pi} \int_{\mathbb{R}^2 \times
    \mathbb{R}^2} \log \left( \frac{1}{|x-y|} \right) |u(x)|^2
  |u(y)|^2 \, dx\, dy,
\end{gather*}
where we set $V_\varepsilon(x) = V(\varepsilon x)$.

We observe that 
\begin{gather} \label{eq:normaV}
\|u\|^2 = \int_{\mathbb{R}^2} \left[ |\nabla u|^2 +
  V_\varepsilon |u|^2 \right] \, dx + |u|_*^2
\end{gather}
can be considered as an equivalent norm on $X$ by virtue of assumption (V).
The functional $I_\varepsilon$ fails to be continuous on the Sobolev
space $H^1(\mathbb{R}^2)$. On the contrary, arguing as in \cite[Lemma 2.2]{CW}, we can infer the following regularity result on $X$.

\begin{proposition}
  If $V$ satisfies (V), then $I_\varepsilon$ is a functional of class
  $C^2$ on $X$.
\end{proposition}

\section{Limiting Equation}

We consider the planar integro-differential equation
\begin{align}
  \label{prob2}
  - \Delta u +  u = \frac{1}{2 \pi}  \Bigl[\log \frac{1}{|\cdot|} \star \abs{u}^2 \Bigr] \ u, \qquad \text{in $\R^2$},
\end{align}
which has the r\^{o}le of a \emph{limiting problem} for \eqref{key56}.
We define the energy functional $I\colon X \to \R$ associated to
\eqref{prob2}:
\begin{align*}
  I (u) = \frac{1}{2} \|u\|_{H^1}^2 +  \frac{1}{8 \pi} \int_{\R^2 \times
  \R^2} \log(|x-y|) |u(x)|^2 |u(y)|^2 \,dx \, dy .
\end{align*}
For future reference, we introduce some shorthand: let us set
\begin{align*}
  B(f,g) = -\frac{1}{2\pi} \int_{\mathbb{R}^2 \times \mathbb{R}^2} \log
  |x-y|\ f(x) g(y) \, dx\, dy,
\end{align*}
so that
\begin{align*}
  I(u) = \left\| u \right\|_{H^1}^2 - \frac{1}{4} B(u^2,u^2).
\end{align*}
It follows from  \cite[Lemma 2.2]{CW}  that $I$ is of class $C^2$
and that
\begin{align*}
  I'(u)[\varphi] &= \int_{\mathbb{R}^2} \left[ \nabla u \cdot \nabla \varphi
                   + u \varphi\right] - B(u^2,u\varphi) \\
  I''(u)[\varphi,\psi] &= \int_{\mathbb{R}^2} \left[ \nabla \varphi \cdot
                         \nabla \psi + \varphi \psi \right] - B(u^2,\varphi
                         \psi) - 2 B(u\varphi,u\psi).
\end{align*}
It has been proved in \cite[Theorem 1.1]{CW} that the restriction of $I$ to the
associated Nehari manifold
\begin{align*}
  \cN = \left\{u \in X \setminus \{0\} \mid I'(u)[u]=0 \right\}
\end{align*}
attains a global minimum. Moreover, every minimizer $u \in \cN$ of
$I_{|\cN}$ is a solution of\eqref{prob2} which does not change sign
and obeys the variational characterization
\begin{align*}
  I(u) = \inf_{u \in X} \sup_{t \in \R} I(tu).
\end{align*}

From \cite[Theorem 1.3]{CW} we have the following result.

\begin{theorem} \label{th:3.1} Every positive solution~$u \in X$ of
  \eqref{prob2} is radially symmetric up to translation and strictly
  decreasing in the distance from the symmetry center. Moreover $u$ is
  unique, up to translation in $\R^2$.
\end{theorem}

Moreover, from \cite[Theorem 1]{BCVS},  the sharp asymptotics of the radially symmetric positive solution of \eqref{prob2} are known.
\begin{theorem} \label{th:3.2}
  If $u \in X$ is a radially symmetric positive solution of
  \eqref{prob2}, there exists~$\mu>0$ such that, as $|x| \to +\infty$,
  \begin{gather*}
    u(x) = \frac{\mu +o(1)}{\sqrt{\strut |x|} (\log |x|)^{1/4}} \exp \left( -\sqrt{M} \mathrm{e}^{-1/M} \int_1^{|x| \mathrm{e}^{1/M}} \sqrt{\log s} \, ds \right),
  \end{gather*}
  where $M= (2\pi)^{-1} \int_{\mathbb{R}^2} |u|^2 \, dx$.
\end{theorem}

We consider the linearization on a positive solution $u$ of
\eqref{prob2}. Let
\(\mathcal{L}(u) \colon \tilde{X} \to L^2(\mathbb{R}^2)\) be the
linear operator defined by
\begin{align*} 
  \mathcal{L}(u)\colon \varphi\mapsto -\Delta \varphi + (1 - w) \varphi + 2 u \left( \frac{\log}{2\pi} \star (u \varphi)\right),
\end{align*}
where
\begin{align*}
  w\colon \R^2\to\R, \quad x\mapsto \frac{1}{2 \pi} \int_{\R^2} \log \frac{1}{\abs{x - y}} \abs{u (y)}^2 dy
\end{align*}
and
 \begin{equation}
\tilde{X} = \left\{ \varphi \in X \mid \hbox{for every $\psi \in C_c^\infty(\mathbb{R}^2)$:}
\int_{\mathbb{R}^2} \varphi
\mathcal{L}(u)\psi = \int_{\mathbb{R}^2} f \psi
\right\}
\end{equation}
By standard arguments, one easily shows that $\mathcal{L}(u)$ is a
self adjoint operator acting on \(L^2(\R^2)\) with domain
\(\tilde X\). Also, differentiating the equation \eqref{prob2}, it is
clear that
$\alpha_1 \partial_{x_1} u+ \alpha_2 \partial_{x_2} u \in \ker
\mathcal{L}(u)$ for every $\alpha_1$, $\alpha_2\in\R$.

The following result has been proved in \cite[Theorem
3]{BCVS}.
\begin{theorem} 
	Let $u \in X$ be a positive solution of \eqref{prob2}. Then
	\begin{align*}
	\ker \mathcal{L}(u) =
	\left\{
	\gamma \cdot \nabla u \ | \  \gamma\in\R^2\right\}.
	\end{align*}
\end{theorem}
The functional-analytic properties of the second derivative of \(I\)
will play a crucial r\^{o}le in our analysis.
\begin{lemma} \label{lem:3.1}	
	Let $u \in X$ be a positive solution of \eqref{prob2}.
	The operator \(I''(u)\) is a Fredholm operator of index zero from \(X\) to its dual space \(X^*\).
\end{lemma}
\begin{proof}
  We will actually prove that $I''(u) = A + K$, where $A$ is a
  bounded invertible operator and $K$ is a compact operator on $X$.
	
  Set $c^2 = \frac{1}{2\pi} \int_{\R^2} u^2(y) \,  dy$. For any
  $\varphi \in X$ and $\psi \in X$, we have
  \begin{align*}
    I''(u) [\varphi, \psi] &= 
                                                   \int_{\R^2} \left[ \nabla \varphi(x) \nabla \psi(x) +  \varphi(x) \psi(x) \right] \, dx 
    \\ &\quad {}+ \frac{1}{2 \pi}  \int_{\R^2} \int_{\R^2}  \log|x-y| u^2(y) \varphi(x) \psi(x) \, dx \, dy  \\
                                                 &\quad{}+ \frac{1}{\pi} \int_{\R^2} \int_{\R^2} \log |x-y| u(y) \varphi(y) u(x) \psi(x) \, dy \, dx \\ 
                                                 &=
                                                   \int_{\R^2} \bigl(\nabla \varphi(x) \nabla \psi(x) +  \varphi(x) \psi(x)  + {c^2}   \log(1 + |x|) \varphi(x) \psi(x) \bigr) \, dx \\	
                                                 &\quad{}+ \frac{1}{2\pi} \int_{\R^2} \int_{\R^2} \left[\log|x-y| - \log (1 + |x|)\right] u^2(y) \varphi(x)  \psi(x) \, dx \, dy \\  &\quad{}+
                                                                                                                                                                                 \frac{1}{\pi} \int_{\R^2} \int_{\R^2} \log|x-y| u(y) \varphi(y) u(x) \psi(x) \, dx \, dy. 
  \end{align*}	
  We have deduced the decomposition \(I''(u)=A +K\), where the operators
  \(A\) and \(K\) act as follows:
  \begin{align}
    \label{eq:14}
    \langle A \varphi, \psi \rangle  =
    \int_{\R^2} \left( \nabla \varphi \cdot \nabla \psi +  \varphi \psi  + {c^2}
    \log(1 + |x|) \varphi(x) \psi(x) \right) dx
  \end{align}
  and
  \begin{align}
     \label{eq:K}
    \langle K \varphi, \psi \rangle & = 
                                          \frac{1}{2\pi} \int_{\R^2} \int_{\R^2} [\log|x-y| - \log (1 + |x|)] u^2(y) \varphi(x)  \psi(x) \, dx \, dy \\  \nonumber &\quad{}+
   \frac{1}{\pi} \int_{\R^2} \int_{\R^2} \log|x-y| u(y) \varphi(y)                                                                            u(x) \psi(x)\, dx\, dy.
  \end{align}
  Equation~\eqref{eq:14} implies that the correspondence
  \begin{gather*}
    u \in X \mapsto \langle A u , u \rangle
  \end{gather*}
  is an equivalent norm on $X$. It follows that the operator $A$ is invertible
  from $X$ to $X^*$.
	
  \smallskip
	
 We claim that $K$ is compact from $X$ to $X^*$. Indeed, let $\{\varphi_n\}_n \subset X$ be a sequence such that
  $\varphi_n \rightharpoonup 0$ as $n \to + \infty.$ It
  follows that $\|\varphi_n\|_X \leq D$ for any $n \in \N$.
	
  We prove that
  \begin{align} \label{eq:13} \lim_{n \to +\infty}
    \sup_{\substack{\psi \in X \\ \|\psi\|_X =1}} |\langle K
    \varphi_n, \psi \rangle| = 0.
  \end{align}
  Fix $\varepsilon >0$ and $\psi \in X$ such that $\| \psi \|_X
  =1$. Since $u \in X$, there exists $M>0$
  such that
  \begin{align*}
    \frac{D}{2  \pi} \int_{|y| >M} \log(1 + |y|) u^2(y) \, dy <
    \frac{\varepsilon}{4} \quad\text{and}\quad \frac{D}{\pi} \int_{|y|
    >M} u^2(y) \, dy < \frac{\varepsilon}{4}.
  \end{align*}
  We evaluate
  \begin{align*}
    \langle K \varphi_n, \psi \rangle &=
                                        \frac{1}{2\pi} \int_{|y|>M} \int_{\R^2} \bigl[\log(1+|x-y|) - \log (1 + |x|)\bigr] u^2(y) \varphi_n(x)  \psi(x) \, dx\, dy \\ 
                                      &{}+
                                        \frac{1}{2\pi} \int_{|y| \leq M} \int_{\R^2} \bigl[\log(1+|x-y|) - \log (1 + |x|)\bigr] u^2(y) \varphi_n(x)  \psi(x) \, dx\, dy \\ &{}-
  \frac{1}{2\pi} \int_{\R^2} \int_{\R^2} \log \left(1 + \frac{1}{|x-y| } \right)  u^2(y) \varphi_n(x)  \psi(x) \, dx\, dy
    \\ &{}+
         \frac{1}{\pi} \int_{\R^2} \int_{\R^2} \log(1+|x-y|) u(y) \varphi_n(y) u(x) \psi(x) \, dx\, dy \\
                                      &{}- \frac{1}{\pi} \int_{\R^2} \int_{\R^2} \log \left(1 + \frac{1}{|x-y| } \right)  u(y) \varphi_n(y) u(x) \psi(x) \, dx\, dy.
  \end{align*}	
  Recalling the elementary
  inequality~$\log(1 + |x-y|) \leq \log (1+ |x|) + \log (1 +|y|)$ for
  $x \in \R^2$, $y \in \R^2$, we have that
  \begin{align*}
    |\langle K \varphi_n, \psi \rangle| &\leq 	
                                          \frac{1}{2\pi} \int_{|y| >M} u^2(y)\, dy \int_{\R^2}
                                          \bigl[2 \log({1 + |x|}) + \log({1 + |y|})\bigr] |\varphi_n(x)|  |\psi(x)|  \, dx \\
                                        &\quad {}+
                                          \frac{1}{2\pi} \int_{|y| \leq M} u^2(y) \, dy \int_{\R^2 }
                                          \Bigl| \log \left(\frac{1 + |x-y|}{ 1+ |x|}\right) \Bigr|  | \varphi_n(x)|  |\psi(x)|  \, dx \\ 
                                        &\quad {}+
                                          \frac{1}{2\pi} \int_{\R^2} \int_{\R^2} \log \left(1 + \frac{1}{|x-y| } \right)  u^2(y) |\varphi_n(x)|  |\psi(x)| \, dx \, dy
    \\ 
                                        &\quad {}+
                                          \frac{1}{\pi} \int_{\R^2} \int_{\R^2} \log(1+ |x-y|) u(y) |\varphi_n(y)| u(x) |\psi(x)| \, dy \\ 
                                        &\quad {}+
                                          \frac{1}{\pi} \int_{\R^2} \int_{\R^2} \log \left(1 + \frac{1}{|x-y| } \right) u(y) |\varphi_n(y)| u(x) |\psi(x)|\, dx\, dy.
  \end{align*}	
  Firstly, we estimate
  \begin{multline*}
    \frac{1}{2\pi} \int_{|y| >M} u^2(y)\, dy \int_{\R^2} \bigl[2 \log({1
      + |x|}) + \log({1 + |y|})\bigr] |\varphi_n(x)|
    |\psi(x)| \, dx \\
    \leq \left( \frac{1}{\pi} \int_{|y| >M} u^2(y) dy \right)
    \|\varphi_n \|_X \|\psi\|_X + \frac{1}{2\pi} \left( \int_{|y| >M}
      \log(1 + |y|) u^2(y) dy \right) \|\varphi_n
    \|_2 \|\psi\|_2 \\
    \leq \frac{D}{\pi} \left( \int_{|y| >M} u^2(y) dy \right)+
    \frac{D}{2 \pi} \left( \int_{|y| >M} \log(1 + |y|) u^2(y) dy
    \right) \leq \frac{\varepsilon}{2}.
  \end{multline*}
  We claim that for every $M>0$, there exists~$L>0$ such that for any
  $y \in \R^2$ with $|y| \leq M$ and for any $x \in \R^2$ we have
  \begin{align}\label{disu}
    \left| \log \frac{1 + |x-y|}{ 1+ |x|} \right|< L.
  \end{align}
  Indeed for any $x \in \R^2$ and $y \in \R^2$, $|y| \leq M$ we have
  \begin{align*}
    \frac{1 + |x-y|}{ 1+ |x|} 
    \leq 1 +M.
  \end{align*}

  Now take $R = 2M -1>0$, we have that $\frac{M}{ 1+ |x|} < 1/2$ for
  any $x \in \R^2$, and $|x| \geq R$.
	
  It follows that for any \(x \in \R^2\), \(y \in \R^2\) with $|x| \geq |y|$, $|x| \geq R$ and \(|y| \leq M\):
  \begin{align*}
    \frac{1 + |x-y|}{ 1+ |x|} \geq \frac{1 + ||x| -|y||}{ 1+ |x|} \geq 1 - \frac{|y|}{ 1+ |x|} \geq 1 -
    \frac{M}{ 1+ |x|} > \frac{1}{2}.
  \end{align*}
  On the other hand, if $|x | \leq R$:
  \begin{align*}
    \frac{1 + |x-y|}{ 1+ |x|} \geq \frac{1}{ 1+ R} =
    \frac{1}{2M}.
  \end{align*}
  Conversely if $|x| \leq |y|$, we infer that $|x| \leq M$ and
  \begin{align*}
    \frac{1 + |x-y|}{ 1+ |x|} \geq
    \frac{1}{ 1+ M}.
  \end{align*}
  We conclude that there exists $L >0$ such that \eqref{disu} holds.

  It follows that
  \begin{multline*}
    \frac{1}{2\pi} \int_{\R^2} \int_{|y| \leq M } \Bigl| \log
    \left(\frac{1 + |x-y|}{ 1+ |x|} \right) \Bigr| u^2(y)
    |\varphi_n(x)| |\psi(x)| \, dx \, dy \\\leq \frac{L}{2\pi}
    \int_{|y| \leq M} u^2(y) \, dy \int_{\R^2} |\varphi_n(x)|
    |\psi(x)| \,dx \leq
    \frac{L}{2\pi} \left(\int_{|y| \leq M} u^2(y)\, dy\right) \|\varphi_n \|_2 \|\psi\|_2  \\
    \leq \frac{\Gamma L}{2 \pi} \|\varphi_n \|_2 \|\psi \|_X =
    \frac{\Gamma L}{2 \pi} \|\varphi_n \|_2,
  \end{multline*}
  where $\Gamma =\int_{|y| \leq M} u^2(y) dy$.
	
  \medskip
	
  By Hardy-Sobolev-Littlewood inequality we have
  \begin{multline*}
    \frac{1}{2\pi} \int_{\R^2 \times \R^2} \log \left(1 +
      \frac{1}{|x-y| } \right) u^2(y) |\varphi_n(x)|
    |\psi(x)| \, dx\, dy \\
    \leq \frac{1}{2\pi} \int_{\R^2 \times \R^2} \frac{1}{|x-y|}
    u^2(y) |\varphi_n(x)| |\psi(x)| \, dx\, dy \leq {c_1}
    \|u \|^2_{8/3} \|\varphi_n \|_{8/3} \|\psi \|_{8/3} \\
    \leq {c_2}\|u \|^2_{8/3} \|\varphi_n \|_{8/3}
    \|\psi \|_X  = {c_2} \|u \|^2_{8/3} \|\varphi_n
    \|_{8/3}
  \end{multline*}
  where $c_1,c_2>0$ are  suitable constants.  Moreover  we can take $R>0$ such that
  \begin{align*}
    \frac{D}{\pi} \left(\int_{|y| > R} \log(1+|y|)
    u^2(y) dy \right)^{\frac12} \|u\|_2 < \frac{\varepsilon}{4}.
  \end{align*}
  We have
  \begin{align*}
    &\frac{1}{\pi} \int_{\R^2 \times \R^2} \log(1+|x-y|) u(y) |\varphi_n(y)| u(x) |\psi(x)| \, dx\, dy  \\
    & \leq
      \frac{1}{\pi} \int_{\R^2 \times \R^2} \log(1+|x|) u(y) |\varphi_n(y)| u(x) |\psi(x)| \, dx\, dy  \\ 
    &\quad {}+
      \frac{1}{\pi} \int_{\R^2 \times \R^2} \log(1+|y|) u(y) |\varphi_n(y)| u(x) |\psi(x)| \, dx\, dy  \\ 
    &\leq
      \frac{1}{\pi }  \|u\|_2 \|u\|_X  \|\varphi_n\|_2 \| \psi\|_X  \\ 
    &\quad {}+ 
      \frac{1}{\pi} \int_{|y| \leq R} \log(1+|y|) u(y) |\varphi_n(y)| dy \int_{\R^2}  u(x) |\psi(x)| dx  \\ 
    &\quad {}+ 
      \frac{1}{\pi} \int_{|y| > R} \log(1+|y|) u(y) |\varphi_n(y)| dy \int_{\R^2}  u(x) |\psi(x)| dx \\ 
    &\leq 
      \frac{1}{\pi}  \|u\|_X^2 \|\varphi_n\|_2 
      \frac{1}{\pi} \log (1+ R) \|u\|^2_2  \|\varphi_n\|_2  \|\psi\|_2
    \\ 
    &\quad {}+
      \frac{D}{\pi} \left(\int_{|y| > R} \log(1+|y|)
      u^2(y) \,  dy \right)^{1/2}
      \|u\|_2 \|\psi\| _X 
    \\ 
    &\leq
      \frac{1}{\pi} \|u\|_X^2 \|\varphi_n\|_2    
      \frac{1}{\pi} \log (1+ R) \|u\|^2_2  \|\varphi_n\|_2   
    \\ 
    &\quad {}+
      \frac{D}{\pi} \left(\int_{|y| > R} \log(1+|y|)
      u^2(y) \,   dy \right)^{1/2}	\|u\|_2   
    \\ 
    &\leq 
      \frac{1}{\pi} \left(1 + 
      \log (1+ R)  \right) \|u\|_X^2 \|\varphi_n\|_2  +
      \frac{\varepsilon}{4}.
  \end{align*}	
  By the Hardy-Sobolev-Littlewood inequality we have
   \begin{align*}
   & \frac{1}{\pi} \int_{\R^2 \times \R^2} \log \left(1+\frac{1}{|x-y|
      } \right) |\varphi_n(y)| u(y) u(x) |\psi(x)| \, dx\, dy \\ & \leq
    \frac{1}{\pi} \int_{\R^2 \times \R^2} \frac{1}{|x-y|} u(y)
    |\varphi_n(y)| u(x) |\psi(x)| \, dx\, dy \\ &
    \leq \frac{1}{\pi} \|u\|^2_{8/3} \|\varphi_n\|_{8/3} \|\psi\|_X=
    \frac{1}{\pi} \|u\|^2_{8/3} \|\varphi_n\|_{8/3}.
 \end{align*}	
  Finally we conclude that
  \begin{equation}
  \sup_{\substack{\psi \in X \\ \|\psi\|_X =1}}  |\langle K \varphi_n, \psi \rangle| \leq \frac{3\varepsilon}{4} +
    {c_3} \|u \|^2_{8/3} \|\varphi_n \|_{8/3} +
    \frac{\Gamma L}{2 \pi} \|\varphi_n \|_2 + \frac{1}{\pi} \bigl(
    1+ \log (1+ R) \bigr) \|u\|_X^2 \|\varphi_n\|_2 
  \end{equation}
  for some $c_3$ positive constant.
  Taking into account that $X$ is compactly embedded into $L^s(\R^2)$
  for any~$s \in [2, +\infty)$ \cite{CW}, we derive that
  $\|\varphi_n \|_{2} \to 0$ and $\|\varphi_n \|_{{8/3}} \to 0$ as
  $n \to + \infty$.  Therefore there exists $n_0 \in \N$ such that for
  any $n \geq n_0$
  \begin{align*}
    c_3  \|u \|^2_{8/3}
    \|\varphi_n \|_{8/3} + \frac{\Gamma L}{2\pi} \|\varphi_n \|_2 +
    \frac{1}{\pi} \bigl( 1+ \log (1+ R)  \bigr) \|u\|_X^2 \|\varphi_n\|_2 
    < \frac{\varepsilon}{4}.
  \end{align*}	
We derive that
\(\lim_{n \to +\infty} |\langle K \varphi_n, \psi \rangle | =0\),
uniformly with respect to $\psi$.  Therefore $K$ is compact and the
proof is complete.
\end{proof}
\bigskip
\begin{definition}
  In the sequel, we will denote by \(U\) the unique positive solution
  of \eqref{prob2} such that
  \begin{align*}
    U(0) = \max_{x \in \mathbb{R}^2} U(x).
  \end{align*}
\end{definition}

From the non-degeneracy result, we can infer the following convexity
property of $I''(U)$.
\begin{proposition} \label{th:2.3}
  The operator $I''(U)$ has only one
  negative eigenvalue, and therefore there exists $\delta>0$ such that
  \begin{align}\label{convex}
    I''(U)[v,v] \geq \delta \|v\|_X^2
  \end{align}
  for every $v \perp_X \operatorname{span} \left\{ U,\frac{\partial U}{\partial
      x}, \frac{\partial U}{\partial y} \right\}$
  , where $\perp_X$ means orthogonality with respect to the inner product $\langle \cdot \mid \cdot \rangle_X$
  \end{proposition}
\begin{proof}
  Since
  \begin{align*}
    -\Delta U + U + \frac{1}{2\pi} \left[ \log \star |U|^2 \right] U=0,
  \end{align*}
  we find that
  \begin{align*} 
    I''(U)[U,U] = \langle \mathcal{L}(U)U , U \rangle =
    -2 \left( \int_{\mathbb{R}^2} |\nabla U|^2 + \int_{\mathbb{R}^2}
    |U|^2 \right)<0.
  \end{align*}
  Let now $ \varphi \in \ker I''(U)$, namely $\varphi \in X$ and
  $I''(U)\varphi =0$ in $X^*$. It follows that $I''(U)\varphi = 0$
  also in $\widetilde{X}^*$, but $\varphi \in \widetilde{X}$, so that
  $\mathcal{L}(U)\varphi =0$. Hence
  $\varphi \in \operatorname{span} \{\partial_1 U,\partial_2 U \}$.

  On the other hand, if
  $\varphi \in \operatorname{span} \{\partial_1 U,\partial_2 U \}$,
  then $\mathcal{L}(U)\varphi =0$ in $\widetilde{X}^*$.  Let
  $\psi \in X$. By density, $\psi$ is the limit in $X$ of a sequence
  $g_n \in C_0^\infty(\mathbb{R}^2)$. It follows that
  \begin{align*}
    I''(U)[\varphi,\psi]=  \lim_{n \to +\infty}
    I''(U)[\varphi,g_n]= \lim_{n \to +\infty} \langle \mathcal{L}(U)\varphi,g_n \rangle = 0
  \end{align*}
  and thus $\varphi \in \ker I''(U)$. This shows that
  $\ker I''(U) = \operatorname{span} \left\{\partial_1 U,\partial_2 U \right\}$.
        
  Taking into account that $U$ is a Mountain Pass solution, by
  Proposition \ref{th:3.1}, we deduce that there exists $\delta >0$
  such that $(\ref{convex})$ holds.
\end{proof}

\bigskip

\section{The perturbation technique}

We will look for solutions to \eqref{key56} near the
embedded submanifold
\(Z = \left\{ z_\xi \mid \xi \in \mathbb{R}^2 \right\}\), where we set
$z_\xi(x)=U(x-\xi)$. 
	Although the norm of $X$ is not invariant under the group of
	translations defined on $X$ by
	\begin{align*}
		\tau_\xi u \colon x \in\mathbb{R}^2 \mapsto u(x-\xi),
	\end{align*}
	the elementary inequality
	\begin{align*}
		\log \left( 1+|x-y| \right) \leq \log \left( 1+|x|+|y|
		\right) \leq \log \left( 1+|x| \right) + \log \left(
		1+|y| \right)
	\end{align*}
	yields that $u \in X$ and $\xi \in \mathbb{R}^2$ implies
	$\tau_\xi u \in X$.
It follows that $U(\cdot - \xi) = \tau_\xi U \in X$ for every
$\xi \in \mathbb{R}^2$.
The invariance under
translation of $I$ then implies that $Z$ is a manifold of critical points
of $I$.

We will show that each point of $Z$ is an approximate critical point
of $I_\varepsilon$, and that there exists a true critical point of
$I_\varepsilon$ located in a tubular neighborhood of $Z$, provided
$\varepsilon$ is small enough.
\begin{lemma} \label{lem:3.5}
  Let assumption~(V) be satisfied. Then there exists a constant $C>0$ such that, for every
  $\xi \in \mathbb{R}^2$ and every $\ge>0$ sufficiently small, we have
  \begin{align*}
    \| I'_\ge(z_\xi) \| \leq C \left( \ge |\nabla V(0)|+\ge^2 \right).
  \end{align*}
\end{lemma}
\begin{proof}
  Since $z_\xi$ is a critical point of $I$, it follows easily that
  \begin{align*}
    |I'_\ge(z_\xi)[v]|^2 \leq \|v\|_{2}^2 \int_{\mathbb{R}^2} |V(\ge
    x)-1|^2 |z_\xi|^2  \, dx
  \end{align*}
  for any $v \in X$. Using the boundedness of $D^2 V$ and the
  exponential decay of $z_\xi$ at infinity, we can prove easily that
  \begin{gather*}
    \int_{\mathbb{R}^2} |V(\ge x)-1|^2 |z_\xi|^2 \, dx \leq C \ge^2
    |\nabla V(0)|^2 + C \ge^4.
  \end{gather*}
\end{proof}
\begin{proposition}
	\label{prop:3.2}
  There exist a constant $\widetilde{C}>0$ and a constant $M>0$ such
  that for every $\xi \in \mathbb{R}^2$, $|\xi| \leq M$, we have
  \begin{align}
    \label{eq:19}
    I'' (z_\xi)[\varphi,\varphi] \geq \widetilde{C} \|\varphi\|_X^2
  \end{align}
  for every $\varphi \perp_X \left( \operatorname{span} \left\{ z_\xi,\frac{\partial
        z_\xi}{\partial x},\frac{\partial z_{\xi}}{\partial y}
    \right\} \right)$, where $\perp_X$ means orthogonality with respect to the inner product $\langle \cdot \mid \cdot \rangle_X$.
\end{proposition}
\begin{proof}
	For the sake of simplicity we denote here ${\perp_X}$ by ${\perp}$.
  In order to get a contradiction, we suppose that there exists a
  sequence~$\{\xi_n\}_n$ in $\mathbb{R}^2$ such that \(\xi_n \to 0\)
  and there exists a sequence \(\{\varphi_n\}_n \subset X\) such that
  \(\varphi_n \in \left( \operatorname{span} \left\{
      z_{\xi_n},\frac{\partial z_{\xi_n}}{\partial x},\frac{\partial
        z_{\xi_n}}{\partial y} \right\} \right)^{\perp}\),
  \begin{align*}
    &\varphi_n \rightharpoonup \bar{\varphi} \quad\text{in $X$ and in $H^1(\mathbb{R}^2)$} \\
    &\varphi_n \to \bar{\varphi} \quad\text{in $L^2(\mathbb{R}^2)$}, \\
    &\|\varphi_n\|_X =1 \quad\text{for every~$n \in \mathbb{N}$},
  \end{align*}
  and
  \begin{align*}
    I''(z_{\xi_n})[\varphi_n,\varphi_n] \leq \frac{1}{n}.
  \end{align*}
  Assume that $\bar{\varphi} \neq 0$. Then,
  \begin{multline*}
    \frac1n \geq I'' (z_{\xi_n})[\varphi_n,\varphi_n] =
    I''(U)[\varphi_n,\varphi_n] + I''(z_{\xi_n})[\varphi_n,\varphi_n] - I''(U)[\varphi_n,\varphi_n] \\
    \geq I''(U)[\varphi_n,\varphi_n] - \left\| I''(z_{\xi_n})-I''(U)
    \right\| \left\| \varphi_n \right\|_X^2 = I''(U)
    [\varphi_n,\varphi_n] - o(1)
  \end{multline*}
  as $n \to +\infty$. Indeed, the functional $I''$ is continuous at
  the point \(U\), and the exponential decay of \(U\) at infinity (see
  Theorem \ref{th:3.2}) immediately yields that $z_{\xi_n} \to U$
  strongly in \(X\).
	
  We claim that $\bar{\varphi} \perp U$,
  $\bar{\varphi} \perp \frac{\partial U}{\partial x}$ and
  $\bar{\varphi} \perp \frac{\partial U}{\partial y}$ in $X$. We only
  prove the first orthogonality property, the other two being
  similar. By assumption, we have that $\varphi_n \perp z_{\xi_n}$,
  $\varphi_n \perp \frac{\partial z_{\xi_n}}{\partial x}$,
  $\varphi_n \perp \frac{\partial z_{\xi_n}}{\partial y}$
  for every $n \in
  \mathbb{N}$. Now,
  \begin{gather*}
    \langle \varphi_n \mid U \rangle_X = - \langle \varphi_n \mid z_{\xi_n}-U \rangle_X.
  \end{gather*}
  The right-hand side converges to zero because $z_{\xi_n} \to U$ and
  $\{\varphi_n\}_n$ is a bounded sequence; the left-hand side
  converges to $\langle \bar{\varphi}\mid U\rangle_X$. We conclude
  that $\bar{\varphi} \perp U$ in $X$. In a similar way we can prove
  that $\bar{\varphi} \perp \frac{\partial U}{\partial x}$ and $\bar{\varphi} \perp \frac{\partial U}{\partial y}$.

  As a consequence,
  \begin{gather*}
    0 \geq \liminf_{n \to +\infty} I''(z_{\xi_n})[\varphi_n,\varphi_n]
    \geq \liminf_{n \to +\infty} I''(U)[\varphi_n,\varphi_n]
    \geq I''(U)[\bar{\varphi},\bar{\varphi}]
    \geq \delta \left\| \bar{\varphi} \right\|_X^2.
  \end{gather*}
  Here we have used Theorem \ref{th:2.3} and the fact that the linear
  operator $I''(U)$ is the sum of a lower semicontinuous operator $A$
  and of a compact operator $K$ introduced in \eqref{eq:14} and
  \eqref{eq:K}. This shows that $\varphi=0$.
	
  But now, exactly as before,
  \begin{align*}
    \frac{1}{n} \geq I''(U) [\varphi_n,\varphi_n] - o(1) 
    = \langle A \varphi_n , \varphi_n \rangle + \langle K \varphi_n , \varphi_n \rangle - o(1) 
    \geq C \|\varphi_n\|_X^2 -o(1) \\
    \geq C - o(1),
  \end{align*}
  a contradiction.
\end{proof}

\bigskip
In what follows, for each $z_\xi \in Z$, we denote by $P^\varepsilon_\xi$ the orthogonal projection
of $X$ onto~$\left( T_{z_\xi} Z \right)^\perp$, where  $X$ is endowed with the norm \eqref{eq:normaV} (depending on $\varepsilon$) and $\perp$ is the orthogonality with respect the associated inner product.
We aim to construct, for every
$z_\xi \in Z$, an element $w=w(\varepsilon,\xi) \in \left( T_{z_\xi} Z \right)^\perp$ such that
\begin{align}
P^\varepsilon_\xi I'_\varepsilon (z_\xi+w) =0 \label{eq:16}
\end{align}
and
\begin{align*}
(\operatorname{Id}-P^\varepsilon_\xi) I'_\varepsilon (z_\xi+w)=0.
\end{align*}
Clearly, the point $u_\varepsilon=z_\xi+w(\varepsilon,z_\xi)$ will be
a critical point of $I_\varepsilon$, i.e. a solution to \eqref{key56}.

To solve the auxiliary equation~\eqref{eq:16} we first write
\begin{gather*}
P^\varepsilon_\xi I'_\varepsilon(z_\xi+w) = P^\varepsilon_\xi I'_\varepsilon (z_\xi) + P^\varepsilon_\xi I''_\varepsilon(z_\xi)[w]+R(z_\xi,w).
\end{gather*}
We will show that $R(z_\xi,w)=o(\|w\|)$ uniformly with respect to
$z_\xi \in Z$ for $|\xi|$ bounded. Then we will show that the linear
operator
\begin{gather*}
B_{\varepsilon,\xi} = - \left( P^\varepsilon_\xi I''_\varepsilon(z_\xi) \right)^{-1}
\end{gather*}
exists and is continuous, so that the equation
$P^\varepsilon_\xi I'_\varepsilon (z_\xi+w)=0$ is equivalent to
\begin{align*}
w = B_{\varepsilon,\xi} \left( P^\varepsilon_\xi I'_\varepsilon (z_\xi) + R(z_\xi,w) \right),
\end{align*}
a fixed-point problem in the unknown $w \in \left( T_{z_\xi} Z \right)^\perp$.

\begin{lemma} \label{lem:3.4}
Let \(M\) be the constant introduced in Proposition \ref{prop:3.2}.  For~$\varepsilon$ sufficiently small, the operator
  $L_\xi = P^\varepsilon_\xi \circ I_\varepsilon''(z_\xi) \circ P^\varepsilon_\xi$ is
  invertible, 
  and there exists a
  constant \(C>0\) such that
  \begin{align*}
    \left\| L_\xi^{-1} \right\| \leq C. 
  \end{align*}
for every \(\xi \in \mathbb{R}^2\) with \(|\xi| \leq M\).
\end{lemma}
\begin{proof}
	Let $\xi \in \R^2$, $|\xi | \leq M$. 
	For simplicity we denote here
	$P^\varepsilon_\xi$ by $P_\xi$.
	 We write \(\left( T_{z_\xi} Z \right)^\perp = V_1 \oplus V_2\),
  where
  \begin{align*}
    V_1 &= \operatorname{span} \{P_\xi z_\xi\} \\
    V_2 &= \left( \operatorname{span}\{z_\xi\} \oplus T_{z_\xi} Z \right)^\perp, 
  \end{align*}
  so that $V_1 \perp V_2$.  We claim that for $\varepsilon \to 0^+$ 
  \begin{align} \label{eq:21}
    \left\| z_\xi - P_\xi z_\xi \right\| =
    o(1), \quad I''_\varepsilon(z_\xi)[z_\xi,\cdot] =\left(
      \frac{1}{\pi} \log \star |z_\xi|^2 \right) z_\xi + o(1).
  \end{align}
  It follows from \eqref{eq:21} that
  \begin{align*}
    L_\xi (z_\xi) &= P_\xi \circ I_\varepsilon''(z_\xi)[P_\xi z_\xi] = P_\xi \left( I_\varepsilon''(z_\xi)[z_\xi,\cdot]+o(1) \right) \\
                  &= P_\xi \left( - \left( \frac{1}{\pi} \log \frac{1}{|\cdot|} \star |z_\xi|^2 \right) z_\xi +o(1) \right) \\
                  &=   \left( \int_{\mathbb{R}^2 \times \mathbb{R}^2}  \log |x-y| |z_\xi (x)|^2 |z_\xi (y)|^2 \, dx \, dy \right) z_\xi + o(1).
  \end{align*}
  As a consequence, the operator $L_\xi$, in matrix form with respect
  to the decomposition
  \(\left( T_{z_\xi} Z \right)^\perp = V_1 \oplus V_2\), can be
  written as
  \begin{align*}
    L_\xi = \left[
    \begin{matrix}
      \left( \int_{\mathbb{R}^2 \times \mathbb{R}^2}  \log |x-y| |z_\xi (x)|^2 |z_\xi (y)|^2 \, dx \, dy \right) \operatorname{Id}+o(1) &o(1) \\
      o(1) &A_\xi
    \end{matrix}
             \right]
  \end{align*}
  where the operator~\(A_\xi\) satisfies~$A_\xi \geq C^{-1} \operatorname{Id}$ according to
  \eqref{eq:19} in Proposition \ref{prop:3.2}.

  It now follows from \eqref{eq:13} that $L_\xi$ is
  negative definite on $V_1$ and thus globally invertible on
  $\left( T_{z_\xi} Z \right)^\perp$.  It remains to prove the
  previous claim.
	
  Recalling the definition of $z_\xi(x) = U(x-\xi)$ and the
  exponential decay of $U$ at infinity, we see that
  \begin{align*}
    \langle z_\xi \mid \partial_{\xi_j} z_\xi \rangle &= - \langle z_\xi \mid \partial_{x_j} z_\xi \rangle = -\langle z_\xi \mid \partial_{x_j} z_\xi \rangle_{X} + \int_{\mathbb{R}^2} \left( V(\varepsilon x)-1 \right) z_\xi \partial_{x_j} z_\xi \, dx \\
                                                      &= o(1) \quad\hbox{as $\varepsilon \to 0$}
  \end{align*}
  for every $i \in \left\{ 1,\ldots,n \right\}$. Therefore,
  \( \left\| z_\xi - P_\xi z_\xi \right\| = o(1) \) as
  $\varepsilon \to 0$. This proves the first part of \eqref{eq:21}.
  The second identity is proved as follows: we compute
  \begin{align*}
    I_\varepsilon''(z_\xi)[z_\xi,v] &= I''(z_\xi)[z_\xi,v] +
                                      \int_{\mathbb{R}^2} \left( V_\varepsilon-1 \right) z_\xi v 
  \end{align*}
  and recall that \(z_\xi\) solves
  \begin{align*}
    -\Delta z_\xi + z_\xi = \frac{1}{2\pi} \left[ \log \frac{1}{|\cdot|}
    \star |z_\xi|^2 \right] z_\xi.
  \end{align*}
  Since
  \( \int_{\mathbb{R}^2} \left( V_\varepsilon-1 \right) z_\xi v =
  o(1) \|v\|\) for $\varepsilon$ small, we conclude that, for any $v \in X$, we have
  \begin{align*}
    I_\varepsilon''(z_\xi)[z_\xi,v] = I''(z_\xi)[z_\xi,v]+ \int_{\mathbb{R}^2} \left( V(\varepsilon x)-1 \right) z_\xi v \, dx 
    =  \left\langle \left( \frac{1}{\pi} \log |\cdot|  \star |z_\xi|^2
    \right) z_\xi \mid v \right\rangle + o (1) \|v\|.
  \end{align*}
\end{proof}
\begin{proposition} \label{prop:3.6}
  Let assumption (V) be satisfied. Then for every
  $\ge$ small, there exists a unique
  $w=w(\ge,\xi) \in (T_{z_\xi} Z )^\perp$ with $|\xi| \leq M$ such that
  $I'_\ge (z_\xi+w(\ge,\xi)) \in T_{z_\xi} Z$. The function
  $(\ge,\xi) \mapsto w(\ge,\xi) $ is of class $C^1$ with respect to
  $\xi$, and there holds
  \begin{align}
    \|w(\ge,\xi) \| &\leq C \left( \ge |\nabla V(0)| +\ge^2 \right) \label{eq:2} \\
    \| \partial_\xi w\| &\leq C \left( \ge |\nabla V(0)| + \ge^2 \right) + o(\ge^2). \label{eq:2.1}
  \end{align}
  Moreover, the function $\Theta_\ge(\xi) = I_\ge (z_\xi+w(\ge,\xi))$ is
  of class $C^1$ and the condition $\Theta'_\ge(\xi_0)=0$ implies
  $I'_\ge(z_{\xi_0}+w(\ge,\xi_0))=0$.
\end{proposition}
\begin{proof}
  Let us recall that our aim is to construct a solution
  $w \in (T_{z_\xi} Z )^\perp$ to \eqref{eq:16}. We write
  \begin{align*}
    I_\ge'(z_\xi+w) = I'_\ge(z_\xi) + I_\ge''(z_\xi)[w] + R(z_\xi,w),
  \end{align*}
  where
  \begin{align*}
    R(z_\xi,w) = I_\ge'(z_\xi+w) - I'_\ge(z_\xi) - I_\ge''(z_\xi)[w] .
  \end{align*}
  By the invertibility of
  $L_\xi = P^\varepsilon_\xi \circ I_\ge''(z_\xi) \circ P^\varepsilon_\xi$  (see Lemma
  \ref{lem:3.4}), the function $w$ solves \eqref{eq:16} if and only if
  \begin{align} \label{eq:3.10} w = N_{\ge,\xi}(w),
  \end{align}
  where
  \begin{align*}
    N_{\ge,\xi}(w) = -L_\xi^{-1} \left( P^\varepsilon_\xi \circ I'_\ge (z_\xi) + P^\varepsilon_\xi R(z_\xi,w) \right).
  \end{align*}
  We can now show that, for $\ge$ sufficiently small, equation
  \eqref{eq:3.10} can be solved by means of the Contraction Mapping
  Theorem.

  First of all, understanding the $L^2$-duality, we have
  \begin{align*}
    I'_\ge (z_\xi+w) = -\Delta z_\xi + V_\ge z_\xi - \Delta w + V_\ge w  +\frac{1}{2\pi} \left[ \log \star (z_\xi+w)^2 \right] (z_\xi+w),
  \end{align*}
  \begin{align*}
    I'_\ge(z_\xi) = -\Delta z_\xi + V_\ge z_\xi + \frac{1}{2\pi} \left[ \log \star |z_\xi|^2 \right] z_\xi
  \end{align*}
  and
  \begin{align*}
    I''_\ge (z_\xi)[w] = -\Delta w + V_\ge w + \frac{1}{2\pi} \left[ \log \star |z_\xi|^2 \right] w + \frac{1}{\pi} \left[ \log \star (z_\xi w) \right] z_\xi.
  \end{align*}
  Therefore, again with respect to the \(L^2\)-duality,
  \begin{align*}
    R(z_\xi,w) &= I'_\ge (z_\xi+w) - I'_\ge(z_\xi)  - I''_\ge (z_\xi)[w] \\
               &= \frac{1}{\pi} \left[ \log \star (z_\xi w) \right] w + \frac{1}{2\pi} \left[ \log \star |w|^2 \right] z_\xi + \frac{1}{2\pi} \left[ \log \star |w|^2 \right] w.
\end{align*}
We have
\begin{gather} \label{eq:4.7}
  \|R(z_\xi,w)\| \leq C \left( \|w\|^2 + o(\|w\|^2) \right)
\end{gather}
as $\|w\| \to 0$.

Indeed we have for any $\phi \in  X$
\begin{align*}
&\pi |\langle R(z_\xi,w), \phi \rangle|  \leq
\left| \int_{\R^2 \times \R^2}\log|x-y| z_\xi(x) w(x) w(y) \phi(y) \, dx \, dy \right| \\ 
& +
\left| \int_{\R^2 \times \R^2} \log|x-y|  |w(x)|^2 z_\xi(y) \phi(y) \, dx \, dy \right| + 
\left| \int_{\R^2 \times\R^2} \log|x-y| |w(x)|^2  w(y) \phi(y) \, dx \, dy \right|   \\ 
&\leq 
\int_{\R^2 \times \R^2} [\log(1+ |x|) + \log (1+ |y|)] |z_\xi(x)| |w(x)| |w(y)| |\phi(y)| \, dx \, dy \\
&+   
\int_{\R^2 \times \R^2} [\log(1+ |x|) + \log(1+|y|)] |w(x)|^2 |z_\xi(y)|  |\phi(y)| \, dx \, dy    
\\ 
& +  
\int_{\R^2 \times \R^2}  [\log(1+ |x|) + \log(1+|y|)] |w(x)|^2 |w(y)|  |\phi(y)| \, dx \, dy  \\ 
&  + 
\int_{\R^2 \times \R^2}\log(1 + \frac{1}{|x-y|}) |z_\xi(x)| |w(x)| |w(y)| |\phi(y)| \, dx \, dy \\
&
\int_{\R^2 \times \R^2}\log(1 + \frac{1}{|x-y|}) |w(x)|^2 |z_\xi(y)| |\phi(y)| \, dx \, dy \\ 
&
+
\int_{\R^2 \times \R^2}\log(1 + \frac{1}{|x-y|}) |w(x)|^2 |w(y)| |\phi(y)| \, dx \, dy \\ 
&  \leq  \|w\|_2 \|\phi\|_2 \|z_\xi\|_X \|w\|_X  + \|z_\xi\|_2 \|w\|_2 \|w\|_X \|\phi\|_X 
+
\|z_\xi\|_2 \|w\|_2 \|w\|_X \|\phi\|_X   
\\
& \quad +  \|w\|^2_X \|z_\xi\|_2 \|\phi\|_2 +   \|w\|_2 \|w\|^2_X \|\phi\|_2  +   \|w\|^2_2 \|w\|_X \|\phi\|_X  \\
& +
\int_{\R^2 \times \R^2}  \frac{1}{|x-y|} |z_\xi(x)| |w(x)| |w(y)| |\phi(y)| \, dx \, dy \\
&   +
\int_{\R^2 \times \R^2} \frac{1}{|x-y|} |w(x)|^2 |z_\xi(y)| |\phi(y)| \, dx \, dy \\ 
& +  
\int_{\R^2 \times \R^2}  \frac{1}{|x-y|} |w(x)|^2 |w(y)| |\phi(y)| \, dx \, dy \\
&
\leq  \|w\|_2 \|\phi\|_2 \|z_\xi\|_X \|w\|_X  + \|z_\xi\|_2 \|w\|_2 \|w\|_X \|\phi\|_X 
+ \|z\|_2 \|w\|_2 \|w\|_X \|\phi\|_X   
\\
& +  \|w\|^2_X \|z_\xi\|_2 \|\phi\|_2 +   \|w\|_2 \|w\|^2_X \|\phi\|_2  +   \|w\|^2_2 \|w\|_X \|\phi\|_X   \\& + C \|w\|^2_{8/3}\|\phi\|_{8/3}\ \|z_\xi\|_{8/3} 
+ C  \|w\|^2_{8/3}\|z_\xi\|_{8/3} \|\phi\|_{8/3}\ +  C \|w\|^3_{8/3} \|\phi\|_{8/3} \\
\end{align*}
for some suitable positive constant $C$.
Since the norm in $(\ref{eq:14})$ is equivalent to $\|\cdot\|_X$ and  \(\phi \in X\) is arbitrary, we have 
\begin{equation}
 \| R(z_\xi,w) \| 
\leq C_1 \|z_\xi\| \|w\|^2    + C_2  \|w\|^3
\end{equation}	
for some $C_1,C_2$ positive constants 
and thus we infer $(\ref{eq:4.7})$.
In a similar way we can deduce that 
\begin{equation}\label{marx}
  \|R(z_\xi,w_1)-R(z_\xi,w_2) \| \leq C \left( \|w_1\| +
\|w_2\| + o(\|w_1-w_2\|)
\right) \|w_1-w_2\|
\end{equation}
 Using Lemma~\ref{lem:3.5}, \eqref{eq:4.7} and \eqref{marx}, we find that
\begin{align*}
  \|N_{\ge,\xi}(w)\| &\leq C \left( \ge |\nabla V(0)|+\ge^2 + \|w\|^2 +
                       o(\|w\|^2) \right) \\
  \|N_{\ge,\xi}(w_1)-N_{\ge,\xi}(w_2) \| &\leq C \left( \|w_1\| +
                                           \|w_2\| + o(\|w_1-w_2\|)
                                           \right) \|w_1-w_2\|.
\end{align*}
As a consequence, the operator $N_{\ge,\xi}$ is a contraction on the
closed subset
\begin{align*}
  W_C = \left\{ w \in \left( T_{z_\xi} Z \right)^\perp \mid \|w\| \leq
  C \left( \ge |\nabla V(0)| + \ge^2 \right) \right\},
\end{align*}
provided that $C>0$ is sufficiently large, and $\ge>0$ is sufficiently
small. The Contraction Mapping Theorem yields a unique fixed point
$w=w(\ge,\xi)$ of $N_{\ge,\xi}$ in $W_C$ such that \eqref{eq:2} holds.
The last statements of the
Proposition are proved by a straightforward modification of the
arguments contained in \cite[pp. 129--130]{AM}, so we present only a
sketch of the ideas.

Let us define the map \(H \colon \mathbb{R}^2 \times X \times
\mathbb{R}^2 \times \mathbb{R} \to X \times \mathbb{R}^2\),
\begin{align*}
  H(\xi,w,\alpha,\ge) =
  \begin{pmatrix}
    I'_\ge (z_\xi+w)-\sum_{i=1}^2 \alpha_i \partial_{x_i} z_\xi \\
    \left( \langle w \mid \partial_{x_1} z_\xi \rangle , \langle w \mid
    \partial_{x_2} z_\xi \rangle \right)
  \end{pmatrix} .
\end{align*}
In particular, $w \in (T_{z_\xi} Z )^\perp$ solves the equation $P_\xi
I'_\ge (z_\xi+w)=0$ if and only if $H(\xi,w,\alpha,\ge)=0$. With
estimates similar to those we have shown above, we can prove that
$\frac{\partial H}{\partial (w,\alpha)}(\xi,0,0,\ge)$ is uniformly
invertible in $\xi$ for $\ge$ small enough. By the Implicit Function
Theorem, the map $\xi \mapsto (w_\xi,\alpha_\xi)$ is of class $C^1$.

Differentiating the identity $H(\xi,w_\xi,\alpha_\xi,\ge)=0$ with
respect to $\xi$, we obtain
\begin{align*}
  \frac{\partial H}{\partial \xi} (\xi,w,\alpha,\ge) + \frac{\partial
    H}{\partial (w,\alpha)} (\xi,w,\alpha,\ge) \frac{\partial
    (w_\xi,\alpha_\xi)}{\partial \xi} =0,
\end{align*}
hence
\begin{align*}
  \| \partial_\xi w\| &\leq C \left\| \frac{\partial H}{\partial
    (w,\alpha)} (\xi,w,\alpha,\ge)[\partial_\xi z_\xi,\alpha] \right\|
  \\
  &\leq C \left( \| I''_\ge (z_\xi+w)[\partial_\xi z_\xi] \| +
  |\alpha| + \| w \| \right).
\end{align*}
It now follows easily that \eqref{eq:2} holds.
\end{proof}

\bigskip
\section{The reduced functional}

Following~\cite{AM}, the manifold
\begin{align*}
Z^\ge = \left\{ z_{\xi}+w(\ge,\xi) \mid \xi \in \mathbb{R}^2,\ |\xi| \leq M,\ 
\ \ge \ll 1 \right\}
\end{align*}
is a natural constraint for $I_\ge$, in the sense that any critical
point of $I_\ge$ constrained to $Z^\ge$ is a free critical point of
$I_\ge$. To prove the existence of a critical point of the functional
$I_\ge$, it is therefore sufficient to show that the constrained
functional \(\Theta_\ge \colon \overline{B(0,M)} \subset \mathbb{R}^2 \to \mathbb{R}\) defined by
\begin{align*}
  \Theta_\varepsilon(\xi)= I_\varepsilon (z_{\xi}+ w)
\end{align*}
possesses a critical point. To this aim, we evaluate
\begin{align*}
  \Theta_\ge(\xi) &= I(z_\xi+w) + \frac12 \int_{\mathbb{R}^2} \left(
                  V_\ge -1 \right)|z_\xi +
                  w|^2 \, dx \\
                &= \frac{1}{2} \int_{\mathbb{R}^2} |\nabla (z_\xi+w)|^2 +
                  |z_\xi+w|^2\, dx  \\
                &\quad {} +  \frac{1}{8 \pi} \int_{\R^2}\int_{\R^2}
                  \log |x-y|
                  |z_{\xi}(x)+ w(x)|^2  |z_{\xi}(y)+
                  w(y)|^2 \,\, dx\, dy \\
                &\quad {} + \frac12 \int_{\mathbb{R}^2} \left(
                  V_\ge -1 \right)|z_\xi +
                  w|^2 \, dx \\
                &= I(z_\xi) + \frac12 \int_{\mathbb{R}^2} \left(
                  V_\ge -1 \right)|z_\xi +
                  w|^2 + R_\ge(w),
\end{align*}
where 
\begin{align*} 
  R_\varepsilon(w) &= 
                     \frac{1}{2} \int_{\R^2} \left( |\nabla w |^2  + w^2 \right)\, dx
                     +  \frac{1}{8 \pi} \int_{\R^2}\int_{\R^2}
                     \log |x-y|
                     |w(x)|^2 |w(y)|^2 \,dx \, dy \\
                   & \quad {}+
                     \int_{\R^2} \left( \nabla z_{\xi} \cdot \nabla w
                       + z_{\xi}w \right)
                     \, dx \nonumber \\
                   & \quad {}	+  \frac{1}{2 \pi} \int_{\R^2}\int_{\R^2}
                     \log |x-y| z_{\xi}(x) w(x)
                     |z_{\xi}(y)|^2
                     \,dx
                     \,
                     dy \nonumber
  \\ 
                   &\quad {}+  \frac{1}{2 \pi} \int_{\R^2}\int_{\R^2}
                     \log |x-y|
                     z_{\xi}(x) w(x)
                     z_{\xi}(y) w(y)\,dx \, dy \nonumber
  \\
                   & \quad {}+  \frac{1}{2 \pi} \int_{\R^2}\int_{\R^2}
                     \log |x-y|
                     z_{\xi}(x) w(x)   |w(y)|^2 \,dx \, dy. \nonumber 
 \end{align*}
 According to Proposition~\ref{prop:3.6}, the function $\Theta_\ge$ can
 be expanded as
 \begin{align} \label{eq:1}
   \Theta_\ge(\xi) = b_0 + \frac12 \int_{\mathbb{R}^2} \left( V(\ge x)-1
   \right) |z_\xi+w|^2 \, dx + o(\ge^2),
 \end{align}
 where $b_0 = I(z_\xi) = I(U)$.
 Let us define $Q_2 = D^2 V(0)$ and the function $\Gamma \colon \mathbb{R}^2 \to
 \mathbb{R}$,
 \begin{align*}
   \Gamma (\xi) = \int_{\mathbb{R}^2} Q_2(x)
   |z_\xi(x)|^2 \, dx.
 \end{align*}
 From now on, we will suppose for the sake of definiteness that
 $x_0=0$ is a proper local minimum of $V$, so that $D^2 V(0)$ is a
 positive-definite quadratic form. The case of a proper local maximum
 can be treated analogously.
 \begin{lemma}
   The point $\xi=0$ is a strict local minimum for $\Gamma$.
 \end{lemma}
 \begin{proof}
By oddness, $\partial_1 \partial_2 \Gamma(0)=0$. Since $\nabla Q_2(x) \cdot x = 2 Q_2(x)>0$, we conclude that $D^2 \Gamma(0)$ is positive-definite.
 \end{proof}
 We fix a number~$\bar{\xi}>0$ in such a way that $\bar{\xi} < M$ and
 \begin{align*}
   \Gamma(\xi) > \Gamma(0)
 \end{align*}
 for every $\xi \in \overline{B} \setminus \{0\}$, where $B =
 B(0,\bar{\xi})$.
 \begin{lemma}
   For $\ge>0$ sufficiently small, there results $\Theta_\ge(0) <
   \inf_{|\xi|=\bar{\xi}} \Theta_\ge(\xi)$.
 \end{lemma}
 \begin{proof}
   We recall the asymptotic expansion (\ref{eq:1}) and observe that
   \begin{align*}
     \lim_{\ge \to 0} \frac{1}{2\ge^{2}} \int_{\mathbb{R}^2} \left( V_\ge -1 \right)
     |z_\xi+w|^2 \, dx = \frac{1}{2} \int_{\mathbb{R}^2} Q_2 |z_\xi|^2
     \, dx = \frac{1}{2} \Gamma(\xi).
   \end{align*}
   Hence
   \begin{align*}
     \Theta_\ge(\xi) - \Theta_\ge (0) = \frac{1}{2} \ge^2 \left(
     \Gamma(\xi)-\Gamma(0) \right) + o(\ge^2).
   \end{align*}
   It now follows from the choice of $\bar{\xi}$ that $\Theta_\ge(\xi) -
   \Theta_\ge (0) > 0$ if $|\xi| = \bar{\xi}$ and $\ge>0$ is small
   enough. The proof is complete.
 \end{proof}
 \begin{proof}[Proof of Theorem \ref{th:main}]
   We have just shown that the function $\Theta_\ge$ must have a minimum
   at some $\xi = \xi(\ge)$ in the ball $B \subset B(0,M)$. This gives rise to a
   critical point $u_\ge = z_\xi+w(\ge,\xi) \in Z^\ge$ of the functional $I_\ge$
   with $\ge \sim 0$. Now, for every $\xi \in \overline{B}$,
   \begin{align*}
     0 \leq \Theta_\ge(\xi)-\Theta_\ge(\xi(\ge)) = \frac{1}{2} \ge^2
     \left( \Gamma(\xi) - \Gamma(\xi(\ge) \right) + o(\ge^2);
   \end{align*}
   as $\ge \to 0$, we may assume that $\xi(\ge) \to \xi_0$ and we
   obtain $\Gamma(\xi) - \Gamma(\xi_0) \geq 0$ for every $\xi \in
   \overline{B}$. Our choice of $\bar{\xi}$ forces $\xi_0=0$, so that
   $\xi(\ge) \to 0$ as $\ge \to 0$. Hence $u_\ge = z_{\xi(\ge)} +
   w(\ge,\xi(\ge)) \to U$.
   
   Coming back to the system (\ref{sys:s-n-red90}) we obtain the
   existence of pairs of solution $(v_\varepsilon, E_\varepsilon)$
   where
   \begin{align*}
   v_\varepsilon(x) = u_\varepsilon \left( \frac{x}{\varepsilon} \right) \simeq U \left(\frac{x}{\varepsilon} \right)
   \end{align*}
   and
   \begin{align*}
   E_\varepsilon (x)&= \omega \left(\frac{x}{\varepsilon} \right)= 
   -\int_{\R^2} \log \left|\frac{x}{\varepsilon} -y \right|
   |u_\varepsilon(y)|^2 \,dy \\
   &= 
   - \frac{1}{\varepsilon^2}  \int_{\R^2} \log
   \frac{|x-z|}{\varepsilon}
   \left|u_\varepsilon\left(\frac{z}{\varepsilon}\right)
   \right|^2 \,dz \\
   &= 
      \frac{1}{\varepsilon^2}\int_{\R^2} \log \frac{\varepsilon}{|x-z|} |v_\varepsilon(z)|^2 \,dz .
   \end{align*}
   Therefore we have
   $E_\varepsilon (x) =  \frac{1}{\varepsilon^2}\int_{\R^2} \log \frac{1}{|x-z|}
   |v_\varepsilon(z)|^2 \,dz + c_\varepsilon$, with
   $c_\varepsilon =\frac{\log \varepsilon}{\varepsilon^2}\| v_\varepsilon\|^2_2$.
   \end{proof}
 \begin{remark}
   Our Theorem \ref{th:main} can be slightly generalized. Indeed, we
   can assume that the potential $V$ has a non-degenerate critical
   point at some $x_0$, in the sense $\nabla V(x_0)=0$ and there
   exists an integer $m \geq 1$ such that $D^{2m} V(x_0)$ is either
   positive- or negative-definite. The proof then requires only a
   higher-order expansion of $I_\ge(z+w)$ in $\ge$. We omit the
   details for brevity.
 \end{remark}
\bigskip
\bigskip 
{\bfseries Acknowledgments.}
The authors wish to thank the referee  for the careful reading and the comments.

The first author is supported by PDR T.1110.14F (FNRS).
The second author is supported by PRIN 2017JPCAPN
``Qualitative and quantitative aspects of nonlinear PDEs'' and by INdAM-GNAMPA. The third author is supported by INdAM-GNAMPA. 
A part of this work was done during the second author was visiting D\'epartement de Math\'ematique, Universit\'e Libre de Bruxelles. She would like to 
thank D\'epartement de Math\'ematique, Universit\'e Libre de Bruxelles, for their
hospitality and support.

\end{document}